\documentclass[12pt]{article}
\makeatletter
\newcommand{\rmnum}[1]{\romannumeral #1}
\newcommand{\Rmnum}[1]{\expandafter\@slowromancap\romannumeral #1@}
\makeatother
\usepackage{}
\usepackage[numbers,sort&compress]{natbib}

\usepackage{color}
\usepackage{threeparttable}

\usepackage[english]{babel}
\usepackage{diagbox}
\usepackage{t1enc}
\usepackage{amsthm,amsmath,amssymb}
\usepackage{mathrsfs}
\usepackage[mathscr]{eucal}
\usepackage[latin1]{inputenc}
\usepackage[english]{babel}
\usepackage{latexsym}
\usepackage{graphicx}
\usepackage{booktabs}
\usepackage{multirow}
\usepackage{listings}
\newtheorem{theorem}{Theorem}
\newtheorem{corollary}[theorem]{Corollary}
\newtheorem{proposition}[theorem]{Proposition}
\newtheorem{lemma}[theorem]{Lemma}
\newtheorem{remark}[theorem]{Remark}

\def \Spec{\textup{Spec}}

\def \diag{\textup{diag}}

\def \T{\textup{T}}

\begin{document}
	\title{Almost all cographs have a cospectral mate}
	\author{Wei Wang\thanks{Corresponding author. Email address: wangwei.math@gmail.com}\quad\quad Ximei Huang\\
		{\footnotesize $^a$ School of Mathematics, Physics and Finance, Anhui Polytechnic University, Wuhu 241000, P. R. China
		}
	}
	\date{}
	\maketitle
	\begin{abstract}
Complement-reducible graphs (or cographs)  are the graphs formed from the single-vertex graph  by the operations of complement and disjoint union. By combining the Johnson-Newman theorem on generalized cospectrality with the standard tools in the asymptotic enumeration of trees, we show that almost all cographs have a cospectral mate.  This result can be viewed as an analogue to a well-known result by Schwenk, who proved that almost all trees  have a cospectral mate. 
		
	\end{abstract}

	
	
	\section{Introduction}
	A complement-reducible graph (or cograph for short) is a graph defined by the following rules  \cite{corneil}:
	
	(\rmnum{1}) the single-vertex graph  is a cograph;
	
	(\rmnum{2}) if $G$ and $H$ are cographs then their disjoint union $G\cup H$ is a cograph;
	
	(\rmnum{3}) if $G$ is a cograph then its complement $\overline{G}$ is a cograph.\\	
Alternatively, instead of using the  complement operation $\overline{G}$, one can use the  join operation $G\vee H$\, which consists of forming the disjoint union $G\cup H$ and then adding an edge between every pair of a vertex from $G$ and a vertex from $H$.   Cographs have arisen in many disparate areas of mathematics and have been independently rediscovered by various researchers (see e.g. \cite{seinsche,jung,sumner,lerchs}). There are many equivalent characterizations of cographs. For example, a graph is a cograph if and only if it is $P_4$-free. Here, a graph is called $P_4$-free if it does not contain the path $P_4$ as an induced subgraph. 

For a graph $G$ with adjacency matrix $A$, the  spectrum of $G$, denoted by $\Spec(G)$, is the multiset of the eigenvalues of $A$. Two graphs $G$ and $H$ are \emph{cospectral} if $\Spec(G)=\Spec(H)$; and we say $H$ is a cospectral mate of $G$ if, in addition, they are not isomorphic. A graph $G$ is \emph{determined by its spectrum} (DS)  if it has no cospectral mate. For general graphs, Haemers \cite{van} conjectured that almost all graphs are determined by their spectrum. However, for the family of trees,  Schwenk \cite{schwenk} proved the other extreme, that is, almost all trees have a cospectral mate. The main result of this paper is to obtain a Schwenk-like result for the family of cographs. 

\begin{theorem}\label{main}
	Almost all cographs have a cospectral mate.
\end{theorem}
\begin{remark}\normalfont
	Although the spectral properties of cographs have received much attention in recent years (see e.g. \cite{ghorbani,mandal,jacobs}), very little work is devoted to the spectral determination of cographs. In a recent manuscript \cite{allem}, Allem et al.~proved that all regular cographs are determined by their Laplacian spectrum (or equivalently by their adjacency spectrum according to \cite[Proposition 3]{van}).
\end{remark}
 For technical reasons, we shall prove a slightly stronger version of Theorem \ref{main} using the notion of generalized spectrum. For a graph $G$, the \emph{generalized spectrum} of $G$ is the pair $(\Spec(G), \Spec(\overline{G}))$. Two graphs $G$ and $H$  are \emph{generalized cospectral} if  $(\Spec(G), \Spec(\overline{G}))=(\Spec(H),\Spec(\overline{H}))$. We say that $H$ is a generalized cospectral mate of $G$ if they are generalized cospectral but not isomorphic. A graph $G$ is \emph{determined by its generalized spectrum} (DGS) if $G$ has no generalized cospectral mate. We note that by the definition, a generalized cospectral mate of $G$ is necessarily a  cospectral mate of $G$, but the converse is not true in general.  Thus, the following theorem is stronger than Theorem \ref{main}.
 \begin{theorem}\label{main2}
 	Almost all cographs have a generalized cospectral mate.
 \end{theorem}
We mention that  Schwenk's result also holds for the generalized spectrum, that is, almost all trees have a generalized spectral mate \cite{gm}.  The proof of Theorem \ref{main2} is based on the cotree representation of cographs and the theorem of Johnson and Newman concerning generalized cospectrality of graphs.  The overall strategy is similar to  Schwenk's proof and can be divided into two stages. In the first stage, we find a single pair of generalized cospectral graphs, one of which is a cograph.  In the second stage, we show that a sufficiently large cotree almost always  contains any given cotree as a subtree in an appropriate sense. In the last section of this paper, some comparative results are given between the family of cographs and the subfamily of threshold graphs.
\section{Cotree representation and generalized cospectrality}
A \emph{cotree} \cite{biyikoglu} is a rooted unordered tree in which no vertex has exactly one child, and all internal vertices (or non-leaf vertices) are labeled with the sign $\cup$ or $\vee$ such that adjacent ones have different labels. Let $T$ be a cotree and $u$ be any vertex of $T$. The subgraph (with $u$ as the root) induced by $u$ and all of its descendants, denoted by $T_u$, is clearly a cotree (where the labels are  inherited from $T$).  Every cotree $T$ defines a cograph naturally by the following three rules:

(\rmnum{1}) If $T$ has only one vertex (and hence no label),  then $T$ represents the single-vertex graph;

(\rmnum{2})  If the root is labeled with the sign $\cup$ and has children $u_1,u_2,\ldots,u_d$ with $d\ge 2$,  then $T$ represents the disjoint union $\bigcup_{i=1}^{d} G_i$, where $G_i$ is the cograph corresponding to the cotree $T_{u_i}$;

(\rmnum{3})  If the root is labeled with the sign $\vee$ and has children $v_1,v_2,\ldots,v_s$ with $s\ge 2$, then $T$ represents the join $\bigvee_{i=1}^{s} H_i$, where $H_i$ is the cograph corresponding to the cotree $T_{v_i}$.

It is known that the cotree representation of a cograph is unique, and two cographs are isomorphic if and only if they correspond to isomorphic cotrees. In order to construct graphs that are cospectral with a cograph, we introduce the notion of quasi-cotrees. A \emph{quasi-cotree} is obtained from a cotree by labeling exactly one external vertex (leaf vertex) with a sign $\star$, where $\star$ represents some graph that is not necessarily a cograph.  Similarly, we can define a graph by a quasi-cotree naturally, although different  quasi-cotrees may define the same graph.  See Figure \ref{qc} for an illustration.
	\begin{figure}
	\centering
	\includegraphics[height=3.5cm]{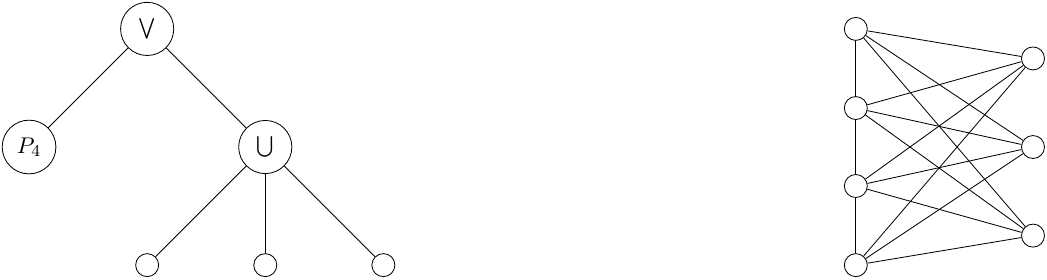}
	\caption{A quasi-cotree and the corresponding graph}
	\label{qc}
\end{figure}
 
 An orthogonal matrix is \emph{regular} if the sum of each row equals one. The following matrix characterization of generalized cospectrality is crucial.
\begin{theorem}[\cite{johnson1980JCTB}]\label{jn}
Let $G$ and $H$ be two graphs. Then $G$ and $H$ are generalized cospectral if and only if there exists a regular orthogonal matrix $Q$ such that $Q^\T A(G) Q=A(H)$.
\end{theorem}  
 \begin{corollary}\label{gcju}
 Let $G_1$ and $G_2$ be generalized cospectral graphs. Then for any graph $H$, the two graphs $G_1\cup H$ and $G_2\cup H$ are generalized cospectral. The same result also holds for the join operation.
  \end{corollary}
  \begin{proof}
  	As $G_1$ and $G_2$ are generalized cospectral, Theorem \ref{jn} implies that there exists a regular orthogonal matrix $Q_0$ such that $Q_0^\T A(G_1)Q_0=A(G_2)$. Define the block diagonal matrix $Q=\diag[Q_0,I]$, where $I$ is the identity matrix of order $|V(H)|$. Clearly, $Q$ is a regular orthogonal matrix and $Q^\T A(G_1\cup H) Q=A(G_2\cup H)$. Using Theorem \ref{jn} again, we conclude that $G_1\cup H$ and $G_2\cup H$ are generalized cospectral.
  	
  	Note that two graphs are generalized cospectral if and only if their complements are generalized cospectral. Thus $G_1\vee H$ and $G_2\vee H$ are generalized cospectral as $G_i\vee H=\overline{\overline{G_i}\cup \overline{H}}$ for $i=1,2$. This completes the proof of Corollary \ref{gcju}.
    \end{proof}
    Figure \ref{cm} gives a pair of generalized cospectral graphs, where the left graph is a cograph (as it is $P_4$-free), whereas the right graph is not. We found this pair using an exhaustive search on all graphs of order $9$, and the left graph in Figure \ref{cm} is the unique (up to complement) cograph of order 9 that is not DGS. 
	\begin{figure}
	\centering
	\includegraphics[height=3.5cm]{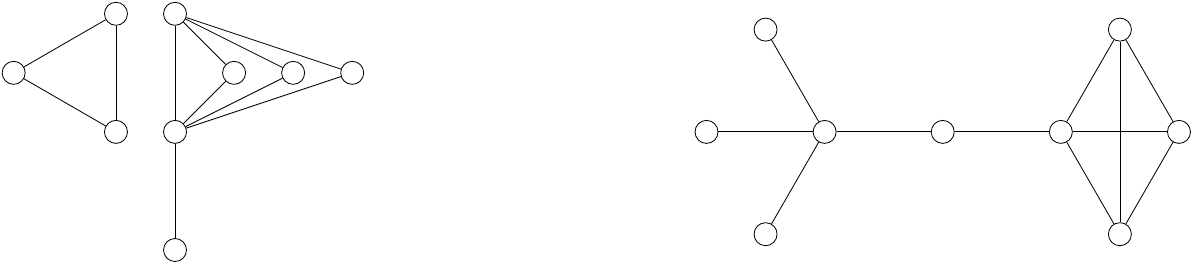}
	\caption{A non-DGS cograph and its generalized cospectral mate}
	\label{cm}
\end{figure}
	\begin{theorem}
		Let $T^*$ be the cotree representation of the left graph in Figure \ref{cm}. Suppose $G$ is a cograph and $T$ is its cotree representation. If there exists an internal vertex $u$ of $T$ such that the subtree $T_u$ is isomorphic to $T^*$, then $G$ has a generalized cospectral mate.
	\end{theorem}
	\begin{proof}
		Let $T'$ be the quasi-cotree obtained from $T$ by replacing the subtree $T_u$ with a distinguished external vertex which represents the right graph in Figure \ref{cm}. Let $G'$ be the graph defined by the quasi-cotree. Clearly, $G'$ is not $P_4$-free and hence is not isomorphic to $G$. On the other hand, using Corollary \ref{gcju} iteratively, we easily see that $G$ and $G'$ are generalized cospectral. Thus, $G'$ is a generalized cospectral mate of $G$, completing the proof.
	\end{proof}

	\section{Asymptotic results on hierarchies}
	
	Let $G$ be a cograph and $T$ be its cotree. Then the structure of the cotree $\hat{T}$ corresponding to the complement $\overline{G}$ is immediate. Indeed, $\hat{T}$ is obtained from $T$ by interchanging all labels for the internal vertices. Note that for any cotree, the labels for all internal vertices are completely determined by the label of its root. Since the DGS-property of a graph is clearly unchanged under the operation of complement, the label of the root is inessential when we are only concerned with the DGS-property.  Following \cite{flajolet}, the underlying rooted tree of a cotree (i.e., removing all labels $\cup$ and $\vee$) is called a \emph{hierarchy}, whose \emph{size} refers to the number of external vertices (leaves).  Figure \ref{h2} illustrates a hierarchy of size $9$, together with two associated cotrees. Note that these two cotrees represent the cograph given in Figure \ref{cm} (left) and its complement, respectively.
	
	Let $\mathcal{H}$ be the combinatorial class (or simply class) of hierarchies, and $H(x)$ be the ordinary generating function of $\mathcal{H}$. Let $\mathcal{Z}$ be the atomic class comprising the single hierarchy of size 1. For a combinatorial class $\mathcal{A}$, the notation $\textup{MSET}(\mathcal{A})$ means the combinatorial class obtained by forming all finite multisets of elements from $\mathcal{A}$. Furthermore, for an integer $k$,  $\textup{MSET}_{\ge k}(\mathcal{A})$ is a subclass of $\textup{MSET}(\mathcal{A})$, meaning that only multisets with at least $k$ elements (from $\mathcal{A}$) are gathered.  Now, the class $\mathcal{H}$ of hierarchies can be  specified by 
	\begin{equation}
	\mathcal{H}=\mathcal{Z}+\textup{MSET}_{\ge 2}(\mathcal{H}).
	\end{equation}
	It follows that 
	\begin{equation}
	H(x)=x+\left(\exp\left(\sum_{k=1}^{\infty}\frac{1}{k}H(x^k)\right)-1-H(x)\right),
	\end{equation}
	or equivalently,
		\begin{equation}
		2H(x)+1-x=\exp\left(\sum_{k=1}^{\infty}\frac{1}{k}H(x^k)\right).
	\end{equation}
		\begin{figure}
		\centering
		\includegraphics[height=7.5cm]{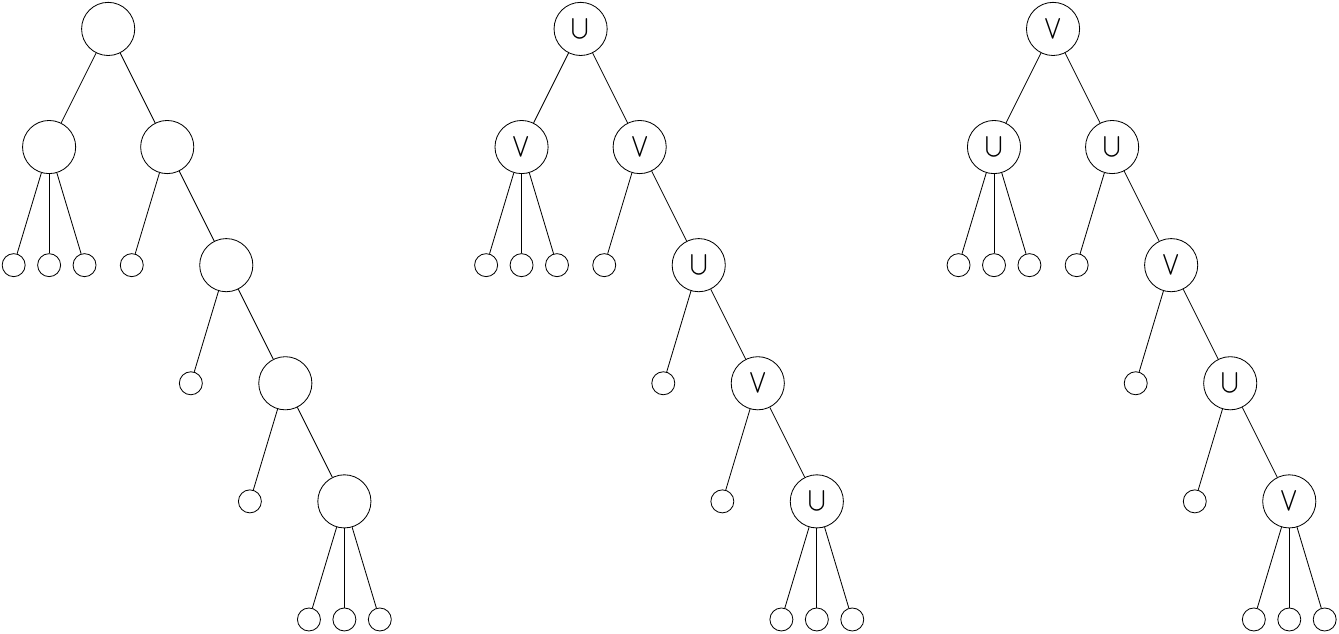}
		\caption{A hierarchy of size 9 and the two associated cotrees}
		\label{h2}
	\end{figure}
	(See \cite{cayley,flajolet} for details)
	
For a hierarchy $T$, a subhierarchy means the hierarchy induced by a vertex and all its descendants (if any). 	Let $T_m$ be any hierarchy of size $m$, where $m\ge 2$. We use $\mathcal{H}^{(T_m)}$ to denote the class of hierarchies that does not contain $T_m$ as a subhierarchy. Suppose that $v_1, u_2,\ldots,u_k$  are all children of the root of $T_m$. Let $S$ be the multiset of $k$ subhierarchies corresponding to $u_1,\ldots,u_k$. Noting that $k\ge 2$ and none of these $k$ hierarchies contains $T_m$ as a subhierarchy, we have $S\in \textup{MSET}_{\ge 2}(\mathcal{H}^{(T_m)})$. Moreover, it is not difficult to see that 
\begin{equation}\label{hm}
	\mathcal{H}^{(T_m)}=\mathcal{Z}+\left(\textup{MSET}_{\ge 2}(\mathcal{H}^{(T_m)})-\{S\}\right).
\end{equation}
Let $H^{(T_m)}(z)$ be the ordinary generating function of the class $\mathcal{H}^{(T_m)}$. Note that the size of $S$ is defined to be the sum of the sizes of the $k$ subhierarchies, which is $m$. This means that the generating function of the class $\{S\}$ is exactly $x^m$. It follows from \eqref{hm} that
\begin{equation*}
	H^{(T_m)}(x)=x+\left(\exp\left(\sum_{k=1}^{\infty}\frac{1}{k}H^{(T_m)}(x^k)\right)-1-H^{(T_m)}(x)-x^m\right),
\end{equation*}
i.e., 
\begin{equation}\label{htm}
	2H^{(T_m)}(x)+1-x+x^m=\exp\left(\sum_{k=1}^{\infty}\frac{1}{k}H^{(T_m)}(x^k)\right).
\end{equation}
We observe from Eq.~\eqref{htm} that the generating functions $H^{(T_m)}(z)$ only depend on the number $m$ but not on the structure of $T_m$.  In other words, if $T$ and $T'$ are two hierarchies of the same size $m$ with $m\ge 2$, then for any $n$, the number of hierarchies of the size $n$ that avoid $T$, is equal to  the number of hierarchies of the same size that avoid $T'$. Thus, in the following, we shall simplify the notation $H^{(T_m)}(x)$ as $H^{(m)}(x)$. 

Let $H_n$ and $H^{(m)}_n$ denote the coefficients of $x^n$ in $H(x)$ and $H^{(m)}(x)$, respectively. The asymptotic behavior of the sequences $\{H_n\}$ and $\{H_n^{(2)}\}$ was determined by Moon \cite{moon}; the asymptotic result for the sequence $\{H_n\}$ also appeared in \cite{ravelomanana}. For two sequences $\{a_n\}$ and $\{b_n\}$,   we write \(a_n \sim b_n\) to denote that they are asymptotically equivalent, i.e., $\lim_{n \to \infty} a_n/b_n = 1.$
\begin{proposition}[\cite{moon,ravelomanana}]\label{hh2}
The numbers $H_n$ and $H^{(2)}_n$ satisfy 
\begin{equation}
	H_n\sim C_0\rho_0^{-n}n^{-3/2} \quad \text{and}\quad H^{(2)}_n\sim C_2\rho_2^{-n}n^{-3/2}, 
\end{equation}
where $C_0=0.2063...$, $\rho_0=0.2808...$, $C_2=0.1972...$ and $\rho_2=0.3462...$. 
\end{proposition} 
\begin{remark}\normalfont
	The two sequences $\{H_n\}$ and $\{H_n^{(2)}\}$ appear as OEIS A000669 and A058385, respectively \cite{oeis}. The first few items of $\{H_n\}$ and $\{H_n^{(2)}\}$ are
	$1$, $1$, $2$, $5$, $12$, $33$, $90$, $261$, $766$, $2312$, $7068$, $21965$, $68954$, $218751$, $699534$ and
	$1$, $0$, $1$, $2$, $4$, $9$, $20$, $47$, $112$, $274$, $678$, $1709$, $4346$, $11176$, $28966$.
\end{remark}
The main tool in proving Proposition \ref{hh2} is the following theorem.
\begin{theorem}[\cite{bender,harary}]\label{basic}
		Let $F(x,y)$ be analytic in each variable separately in some neighborhood of $(x_0,y_0)$ and suppose that the following conditions are satisfied:\\		
	\textup{(\rmnum{1})} $F(x_0,y_0)=0$;\\
	\textup{	(\rmnum{2})} $y=f(x)$ is analytic in $|x|<|x_0|$ and $x_0$ is the unique singularity on the circle of convergence;\\
	\textup{	(\rmnum{3})} if $f(x)=\sum_{n=0}^{\infty}f_nx^n$ is the expansion of $f$ at the origin, then $y_0=\sum_{n=0}^{\infty}f_nx_0^n$;\\
	\textup{	(\rmnum{4})} $\frac{\partial{F}}{\partial{y}}=0$;\\
	\textup{	(\rmnum{5})} $\frac{\partial^2 F}{\partial y^2}\neq 0$.
		
		Then $f(x)$ may be expanded about $x_0$:
		\begin{equation}
			f(x)=f(x_0)+\sum_{k=1}^{\infty}a_k(x_0-x)^{k/2}
		\end{equation}
		and if $a_1\neq 0$, 
		\begin{equation}
			f_n\sim \frac{-a_1}{2\sqrt{\pi}}x_0^{-n+1/2}n^{-3/2}.
		\end{equation}
	\end{theorem}
Now we give the asymptotic result for $H_n^{(m)}$ for any $m\ge 2$, which will implies that $\lim_{n\to \infty}H_n^{(m)}/H_n=0$, and in particular $\lim_{n\to \infty}H_n^{(9)}/H_n=0$.  Indeed, for this purpose, it suffices to estimate the radius for $H^{(m)}(x)$.
\begin{theorem}\label{rd}
Let $m\ge 2$. Then the radius of convergence for $H^{(m)}(x)$ is larger than the radius for $H(x)$.
\end{theorem}
\begin{proof}
We imitate Otter's method \cite{otter} as presented in \cite[Section 9.5]{harary} (see also \cite{harary1975}) for estimating the radius corresponding to ordinary trees. Let $\rho_m$ be the radius of convergence for $H^{(m)}(x)$. Since each coefficient satisfies $0\le H^{(m)}_n\le H_n$, we know that $\rho_m\ge \rho_0$, where $\rho_0=0.2808...$ is the radius of $H(x)$.
\begin{lemma}\label{lcon}
	The limit of $H^{(m)}(x)$ as $x$ approaches $\rho_m$ from the left exists and is equal to $\sum_{k=1}^{\infty}H_n^{(m)}\rho_m^k$.
\end{lemma}
\begin{proof}
	Since $H^{(m)}(x)$ satisfies the functional equation  \eqref{htm}, we have for all $x\in (0,\rho_m)$,
	\begin{equation}
		2H^{(m)}(x)+1-x+x^m=\exp\left(H^{(m)}(x)+H^{(m)}(x^2)/2+\cdots\right)>\exp\left(H^{(m)}(x)\right).
	\end{equation}
	Let $L=1+\rho_m+\rho_m^m$. Then clearly $L$ is an upper bound of the function $f(x)=1-x+x^m$ on the interval $(0,\rho_m)$. It follows that
	$$2H^{(m)}(x)+L>\exp\left(H^{(m)}(x)\right),~\textup{for} ~ x\in (0,\rho_m).$$
	Hence, $H^{(m)}(x)$ is bounded on the interval $(0,\rho_m)$. Since $H^{(m)}(x)$ is monotonic, the left-hand limit at $\rho_m$ exists, and we let $b_m$ denote this limit. It now follows quickly that $b_m=H^{(m)}(\rho_m)$.
\end{proof}
\begin{lemma}\label{rm}
	The series $H^{(m)}(x)$ satisfies 
\begin{equation}
	H^{(m)}(\rho_m)=\frac{1+\rho_m-\rho_m^m}{2}.
\end{equation}
	\end{lemma}
\begin{proof}
	The generating function $H^{(m)}(x)$  satisfies the functional relation $F(x,H^{(m)}(x))=0$, where
	\begin{equation}\label{Fxy}
		F(x,y)=\exp\left(y+\sum_{k=2}^{\infty}H^{(m)}(x^k)/k\right)-2y-1+x-x^m.
	\end{equation}
	We observe from Eq.~\eqref{htm} that  $y=H^{(m)}(x)$ is the unique analytic solution of $F(x,y)=0$ and it has a singularity at $x=\rho_m$. 
	On differentiating Eq.~\eqref{Fxy} with respect to $y$, we find
	$$\frac{\partial F}{\partial y}=F(x,y)+2y-1-x+x^m.$$
	By Lemma \ref{lcon}, we have $F(\rho_m,b_m)=F(\rho_m,H^{(m)}(\rho_m))=0$. Consequently, the partial derivative $\partial F/\partial y$ at $(\rho_m,b_m)$ is given by
	$$\frac{\partial F}{\partial y}(\rho_m,b_m)=2b_m-1-\rho_m+\rho_m^m.$$
	Furthermore, by the implicit function theorem and the established fact that $x=\rho_m$ is a singularity,  this partial derivative must be zero at $(\rho_m,b_m)$, i.e., $b_m=(1+\rho_m-\rho_m^m)/2.$ This completes the proof of Lemma \ref{rm}.
\end{proof}

	Lemma \ref{rm}, combining with Eq. \eqref{htm} for $x=\rho_m$, leads to 
	\begin{equation}
		\exp\left(\sum_{k=1}^{\infty}\frac{1}{k}H^{(m)}(\rho_m^k)\right)=2.
	\end{equation}
A	similar argument shows that the corresponding equation for $\rho_0$ is 
	\begin{equation}
		\exp\left(\sum_{k=1}^{\infty}\frac{1}{k}H(\rho_0^k)\right)=2.
	\end{equation} 
(See also \cite{moon}.)  Since both series have nonnegative coefficients, and since the coefficients $H^{(m)}_n=H_n$ for $n\le m-1$ and $H_n^{(m)}<H_n$ for $n\ge m$, we conclude that $\rho_m$ is strictly greater than $\rho_0$, completing the proof of Theorem \ref{rd}.
\end{proof}
We are now in a position to prove Theorem \ref{main2}.

\noindent\textbf{Proof of Theorem \ref{main2}}  Let $T_9$ be the hierarchy as shown in Figure \ref{h2}. Using Theorem \ref{rd} for $m=9$, we see that the coefficients $H_n^{(9)}$ have a slower rate of growth than the coefficients $H_n$. We conclude that
\begin{equation}
	\lim_{n\to \infty}\frac{H_n^{(9)}}{H_n}=0.
\end{equation}
Thus, among all hierarchies of size $n$, almost every hierarchy has $T_9$ as a subhierarchy (when $n$ approaches infinity). Let $K$ be such a hierarchy and $T$, $T'$ be two cotrees obtained from $H$. Suppose that $G$ and $\overline{G}$ are the two graphs represented by $T$ and $T'$. Let $T^*$ be the first cotree in Figure \ref{h2} which represents the non-DGS cograph in Figure \ref{cm}. Clearly, either $T$ or $T'$ will contain $T^*$ as a subtree. It follows from Theorem \ref{main2} that either $G$ or $\overline{G}$ will have a generalized mate. But if $H$ is a generalized cospectral mate of $G$, then $\overline{H}$ is a generalized cospectral mate of $\overline{G}$. This means that both $G$ and its complement $\overline{G}$ have a cospectral mate.
Therefore, almost all cographs have a cospectral mate. \hfill \qed

It should be mentioned that for cographs with a small number of vertices, only a very small portion has a cospectral  mate.  However, Theorem \ref{main} states that, asymptotically, this is the exception. In his classic paper \cite{schwenk} on trees, Schwenk estimated that one must have $n\ge 4919$ to be sure that more than half the trees have a cospectral mate. To obtain a similar estimation in the setting of cographs, one needs to obtain the  coefficients of the asymptotic formula $H^{(m)}_n\sim C_m \rho_m^{-n}n^{-3/2}$ (for $m=9$). Using a similar argument as in \cite{moon} together with the numerical algorithm \cite[\Rmnum{7}.21, p. 477]{flajolet}, one can estimate 
$C_9\approx0.2063663931885738$ and $\rho_9\approx 0.2808383687063348$, whereas $C_0\approx 0.2063814446007890$ and $\rho_0\approx0.2808326669842004$. (The values of $C_0$ and $\rho_0$ for the sequence $\{H_n\}$ (OEIS A000669) were previously reported by Vaclav Kotesovec.) Thus we have 
\begin{equation}
	\frac{H^{(9)}_n}{H_n} \sim \frac{C_9}{C_0}\left(\frac{\rho_0}{\rho_9}\right)^n\approx (0.999979697495892)^n.
\end{equation}
Our method guarantees that among all cographs of order $n$, about $1-H_n^{(9)}/H_n$ of them have a generalized cospectral mate. Therefore, we must have 
$$n\ge \frac{\ln 0.5}{\ln 0.999979697495892}\approx34141$$
to be sure that more than half the cographs have a cospectral mate.

\section{Comparing with threshold graphs}
Threshold graphs are an important and well-studied subclass of cographs.  A graph is a threshold graph if it can be constructed from the single-vertex graph by repeatedly adding either an isolated vertex or a dominating vertex. Recent studies suggest that, from the spectral point of view, there are many similarities between threshold graphs and cographs. For example, threshold graphs have no eigenvalues in the interval $(-1,0)$ \cite{jacobs2015}, and the same conclusion also holds for cographs \cite{mohammadian}. 

We are most interested in the spectral determination of threshold graphs and cographs. Currently, we do not know whether there exist two cospectral graphs such that one is a threshold graph but the other is not. However,  Theorem \ref{main} tells us that for cographs, almost all have a cospectral mate that is not a cograph. The following remarkable theorem of Lazzarin, M\'{a}rquez and Tura \cite{lazzarin} states that when we restrict to the family of threshold graphs, all  are distinguishable by their spectrum.  
\begin{theorem}[\cite{lazzarin}]
	No threshold graphs  have a cospectral mate in the family of threshold graphs.
\end{theorem}
We show that the situation for cographs is quite different.
\begin{theorem}
	Almost all cographs have a generalized cospectral mate in the family of cographs.
\end{theorem}
\begin{proof}
	The proof follows the same line as the proof of Theorem \ref{main2}. The only difficulty here is to find a pair of cospectral graphs in the family of cographs. With the computational implementation of the cograph generator (which is based on the paper \cite{jones} and can be accessed from github.com/atilaajones/CographGeneration), we find that the minimal pair of non-isomorphic  generalized cospectral cographs has 15 vertices. Among 1,399,068 cographs of order 15, there are essentially only $2$ pairs of generalized cospectral mates (the other 2 pairs are obtained by taking complements). Figure \ref{c15} shows one such pair of graphs with graph6 strings  \texttt{N]?GWWGAGP@FAMAM@F?} and \texttt{Ns\_??KF@oK?p@a@b\_po}. By Theorem \ref{rd} for $m=15$,
	the radius of convergence for $H^{(15)}(x)$ is larger than the radius for $H(x)$. Thus, the same argument as in the proof of Theorem \ref{main2} completes the proof of this theorem.
\end{proof}
	\begin{figure}
	\centering
	\includegraphics[height=5cm]{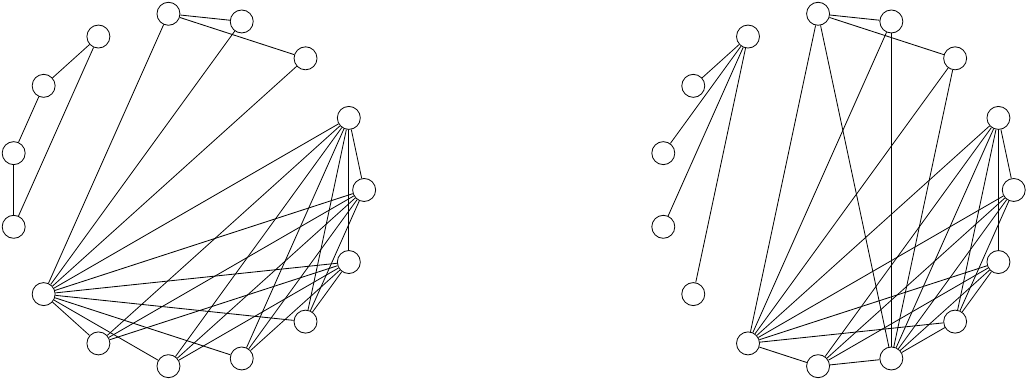}
	\caption{A pair of minimal generalized cospectral cographs}
	\label{c15}
\end{figure}
\begin{remark}\normalfont
	We can obtain $\rho_{15}\approx0.2808326697806751$ and hence \begin{equation}
		\frac{\rho_0}{\rho_{15}}\approx\frac{0.2808326669842004}{0.2808326697806751}\approx0.999999990042203
	\end{equation}which is very close to 1. Thus, using our method of proof, to guarantee that more than half the cographs have a generalized cospectral mate which is also a cograph, the number of vertices should satisfy
	\begin{equation}
		n\ge \frac{\ln 0.5}{\ln 0.999999990042203 }\approx 6.96\times 10^{7}.
	\end{equation} 
\end{remark}
For a graph, the \emph{signless Laplacian matrix} (or $Q$-matrix) is defined to be the sum of the adjacency matrix and the degree matrix. Although it is believed that the $Q$-spectrum is usually more powerful to distinguish graphs than the ordinary (adjacency) spectrum \cite{cvetkovic},  the result is very different for the family of threshold graphs. Indeed,  Carvalho et al. \cite{carvalho} give a simple positive lower bound on the proportion of threshold graphs that have a $Q$-cospectral mate.
\begin{theorem}[\cite{carvalho}]\label{qt}
	For each $n\ge 4$, at least $1/8$ of the threshold graphs of order $n$ have a $Q$-cospectral mate.
\end{theorem}
We say that two graphs are generalized $Q$-cospectral if they are $Q$-cospectral and their complements are also  $Q$-cospectral. It is known that the theorem of Johnson-Newman also holds in the settings of generalized $Q$-spectra \cite{qiu2019}. From this point of view, we can simplify the argument of the cospectrality in \cite{carvalho}, and moreover, we conclude that the $Q$-cospectral mates constructed there  are indeed generalized $Q$-cospectral.  However, the asymptotic analysis does not apply to threshold graphs since  cotrees corresponding to threshold graphs are rather restricted. The number $1/8$ in Theorem \ref{qt} seems to be best possible. However, for cographs we have the following result.
\begin{theorem}
	Almost all cographs have a generalized $Q$-cospectral mate.
\end{theorem}
\begin{proof}
	The proof is direct. Note that the minimal pair of generalized $Q$-cospectral mates consists of the star graph $K_{1,3}$ and its complement, both of which are cographs.  Using Theorem \ref{rd} for $m=4$ we see that the radius of convergence of $H^{(4)}(x)$ is larger than the radius for $H(x)$. Noting that Theorems \ref{jn} and \ref{gcju} also hold for generalized $Q$-spectrum, we are done by a similar argument as in the proof of Theorem \ref{main2}.
\end{proof}

\section*{Acknowledgments}
This work is partially supported by the National Natural Science Foundation of China (Grant No. 12001006), Wuhu Science and Technology Project, China (Grant
No. 2024kj015). We would like to thank Vaclav Kotesovec for his guidance on the computation of the asymptotic coefficients. We also thank MathOverflow user AnttiP for their  helpful program, which provided experimental verification of our result during the preparation of this manuscript.

\end{document}